\documentclass[11pt]{article}
\usepackage[utf8]{inputenc}

\usepackage{latexsym}
\usepackage[all]{xy}
\usepackage{amsmath, amsthm}
\usepackage{amssymb}
\usepackage{amsfonts}
\usepackage{color}
\usepackage{mathtools}
\usepackage[T1]{fontenc}
\usepackage{mathtools}
\usepackage{stmaryrd}
\usepackage{hyperref}
\usepackage{graphicx}
\usepackage{xcolor}
\hypersetup{colorlinks=true,urlcolor=blue,linkcolor=blue,citecolor=blue}

\usepackage{cite}

\newcommand {\Ab} {\mathsf{Ab}}

\newcommand {\QCoh} {\mathsf{QCoh}}

\newcommand {\CAlg} {\mathsf{CAlg}}
\newcommand {\Fib} {\mathsf{Fib}}
\newcommand {\Alg} {\mathsf{Alg}}
\newcommand {\Mod} {\mathsf{Mod}}
\newcommand {\LMod} {\mathsf{LMod}}

\newcommand{\Ind}{\mathsf{Ind}}
\newcommand{\Tan}{\mathcal{T}an}
\newcommand{\Fun}{\mathsf{Fun}}
\newcommand{\Gerbe}{\mathcal{G}erb}

\newcommand{\heart}{\ensuremath\heartsuit}

\newcommand{\ZZ}{\mathbb{Z}}
\newcommand{\NN}{\mathbb{N}}

\newcommand {\Map} {\mathbf{Map}}

\newcommand {\Parf} {\mathsf{Perf}}
\newcommand {\IParf} {\widehat{\mathsf{Perf}}}

\newcommand {\OO} {\mathcal{O}}

\newcommand {\Spec} {\mathbf{Spec}}

\newcommand {\scat} {\s\text{-}\mathbf{Cat}}

\newcommand  {\St}     {\mathbf{St}}
\newcommand  {\Aff}     {\mathbf{Aff}}

\newcommand  {\colim}   {\mathrm{colim}}
\newcommand  {\Sym}   {\mathrm{Sym}}
\newcommand  {\LSym}   {\mathrm{LSym}}

\newcommand{\s}{\infty}

\newcommand{\mcat}{\s\text{-}\mathbf{Cat}^\mu}
\newcommand{\tcat}{\s\text{-}\mathbf{Cat}_{pr}^\Theta}
\newcommand{\TCAlg}{\Theta\text{-}\CAlg}

\theoremstyle{plain}
\newtheorem{thm}{Theorem}[section]
\newtheorem{prop}[thm]{Proposition}
\newtheorem{lem}[thm]{Lemma}

\newtheorem{cor}[thm]{Corollary}

\theoremstyle{definition}
\newtheorem{df}[thm]{Definition}

\newtheorem{rmk}[thm]{Remark}

\newtheorem{nota}[thm]{Notation}

\title{$\Theta$-Categories and Tannakian duality}
\author{Joost NUITEN, Bertrand TOEN}
\date{July 2025}

\begin{document}

\maketitle

\begin{abstract}
We introduce a notion of $\Theta$-categories, which is a refinement of the notion of 
symmetric monoidal $\s$-categories. We use this notion to prove a Tannakian duality statement, 
relating $\Theta$-categories with fpqc-stacks by means of 
a certain stack of fiber functors in the context of $\Theta$-categories. This provides,
over a base ring of arbitrary characteristic, 
a strong link between \emph{Tannakian $\Theta$-categories} and the schematic homotopy types
of \cite{chaff}.
\end{abstract}

\tableofcontents

\section*{Introduction}

The idea that homotopy theory can be approached from a Galois or Tannakian point of view goes back 
to Grothendieck. This perspective was first incarnated in his theory of the fundamental group, both pro-finite and
pro-algebraic (see for instance \cite{zbMATH03370468}, \cite{zbMATH00015170}). To the best of our knowledge, the fact 
that this point of view
on the fundamental group can be fruitfully extended to more general homotopy invariants appears
in two letters of Grothendieck: to Serre, dated 18.07.1959 for the pro-finite case (see \cite[p.~75]{zbMATH01652919}), 
and to Breen, dated 19.07.1975, for the pro-algebraic case (see \cite[cote 134-2, p.~75]{grotharchives}). 

In \cite{tan}, the second author suggested an approach to Tannakian duality for $\s$-categories based
on the notion of symmetric monoidal $\s$-categories defined as functors out of $\Gamma$, the category 
of pointed finite sets. In \cite{chaff} certain stacks, called \emph{schematic homotopy types},
have been identified as the expected duals to 
Tannakian $\s$-categories and have been used notably in the setting of non-abelian Hodge theory 
(see \cite{zbMATH05292808}). 
Thanks to the foundational work of Lurie, actual Tannaka duality statements have been 
proven by several authors, see for instance \cite{LurieTan}, \cite{arXiv:1204.5787} and 
\cite{zbMATH06902492}. However, except in the specific case of characteristic zero, 
these results are obtained in the setting of $E_\s$-algebras, which do not relate well to
the theory of schematic homotopy types of \cite{chaff}, and therefore do not provide 
the expected Tannakian interpretations for schematic homotopy types. \\

The starting point of this work is the simple observation that any notion
of Tannakian $\s$-category solely based on symmetric monoidal $\s$-categories will fail
to relate nicely to schematic homotopy types, or to stacks in algebraic geometry in general. Indeed, for any 
stack $F$ on the fpqc-site of affine schemes, 
the symmetric monoidal $\s$-category $\QCoh(F)$ of quasi-coherent sheaves
comes equipped with natural extra operations, namely the derived symmetric power operations. These  
cannot be reconstructed purely from the symmetric monoidal structure, and their 
existence is a theoretical obstruction for any Tannakian-type reconstruction result in algebraic geometry. 

A second incarnation of this discrepancy can be seen at the level of commutative algebras. The theory of schematic homotopy types from \cite{chaff} uses co-simplicial commutative rings, a sub-class of the $\LSym$-algebras (see e.g.~\cite{Joostsc, raksit2020}) typically used in derived, as opposed to spectral, algebraic geometry. In characteristic zero, these $\LSym$-algebras coincide with the $E_\s$-algebras that exist naturally in any symmetric monoidal $\s$-category, but have an additional algebraic structure in general.

The purpose of this short note is to introduce a rudimentary notion of \emph{$\Theta$-categories}, which is
a refinement of the notion of symmetric monoidal $\s$-categories suited for a Tannakian 
interpretation of schematic homotopy types and similar stacks. The point of view adopted here will be
somewhat minimalist and takes the second discrepancy mentioned above as its starting point. It is based on the principle that $\Theta$-categories should morally be symmetric monoidal $\s$-categories in which an internal notion of $\LSym$-algebras exists. 
We implement this idea by defining $\Theta$-categories simply as pairs $(T,M)$, consisting of
a symmetric monoidal $\s$-category and a monad $M$ over $T$ equipped with a map from the $E_\s$-monad and satisfying 
some conditions (see Definition \ref{d1-3}). The typical example to keep in mind is the case where $F$ is a stack, 
$T=\QCoh(F)$ is its $\s$-category of quasi-coherent complexes and $M$ is the $\LSym$-monad,
given by the direct sum of derived symmetric powers $E \mapsto \bigoplus_n \LSym^n(E)$. 
The main content of our work is that this definition gives rise to a well-behaved $(2,\s)$-category of $\Theta$-categories, from which a Tannakian duality theorem can 
be proven (see Theorem \ref{t1}). In particular, we show that any Tannakian $\Theta$-category $T$ is
equivalent, as a symmetric monoidal $\s$-category, to $\IParf(Fib^{\LSym}(T))$, the $\s$-category
of ind-perfect complexes on a certain stack of $\LSym$-fiber functors. The stacks obtained
as stacks of $\LSym$-fiber functors are moreover essentially the schematic homotopy types of \cite{chaff}. \\

To finish this introduction, we would like to mention that the present approach to $\Theta$-categories
is not optimal. The reader will for instance observe some limitations and 
unnatural conditions in our main result, Theorem \ref{t1}. This is the price to pay for our minimalist 
approach, but we mention some possible modifications and improvements in a final comment (see Remark \ref{r2}).
We also want to mention that a more systematic theory of $\Theta$-categories as symmetric monoidal $\infty$-categories with additional symmetric powers is presently under
investigation by the first author. We believe that the outcome of this work will allow
for stronger and more natural results in a near future. Nevertheless, Theorem \ref{t1} 
can be applied in several important situations to construct interesting Tannakian duals, such 
as the theory of Nori's motives (from which we can extract a motivic homotopy type defined
as a stack over $\Spec\, \ZZ$), or the theory of exponential complexes appearing in \cite{zbMATH05380327} (from 
which we can extract the \emph{exponential homotopy type} associated to an algebraic
variety). These specific schematic homotopy theories will be studied elsewhere.

\section{$\Theta$-Categories}

\subsection{$\mu$-Categories}

We let $\scat_{pr}^{R}$ be the $(2,\s)$-category of presentable $(1,\s)$-categories and 
right adjoints (i.e. limit preserving accessible $\s$-functors). We have the corresponding 
$(2,\s)$-category $\Fun(\Delta^1,\scat^{R}_{pr})$, of morphisms in $\scat_{pr}^{R}$. For any 
object $f \colon T' \to T$ in $\Fun(\Delta^1,\scat^R_{pr})$, we will denote by 
$f^* \colon T \to T'$ the left adjoint to $f$, and sometimes write $f_*$ instead of $f$.

\begin{df}\label{d1-1}
An object $f_* \colon T' \to T$ in $\Fun(\Delta^1,\scat^{R}_{pr})$ is a \emph{$\mu$-category} if 
$f_*$ is conservative and preserves sifted colimits. 
\end{df}

Note that, by the Barr--Beck theorem \cite[Thm.~4.7.0.3]{HA} the condition of being a $\mu$-category 
implies that the canonical factorization of $f_*$
$$T' \to \LMod_{M}(T) \to T,$$
induces an equivalence $T' \simeq \LMod_M(T)$. Here $M=f_*f^*$ is the monad on $T$ associated with the 
adjunction $(f^*, f_*)$, and $\LMod_{M}(T)$ is the $\s$-category of $M$-modules in $T$. 

\begin{df}\label{d1-2}
Let $f_* \colon T_1' \to T_1$  and $g_* \colon T_2' \to T_2$ be two $\mu$-categories in the sense of Definition \ref{d1-1}.
A \emph{$\mu$-morphism} from $T_1' \to T_1$ to $T_2' \to T_2$ consists of a commutative square
in $\scat$
$$\xymatrix{
T_1' \ar[r]^{u}  \ar[d]_-{f_*} & T_2' \ar[d]^-{g_*} \\
T_1 \ar[r]_{v} & T_2,
}$$
with $u$ and $v$ colimit preserving $\s$-functors, and
such that the induced natural transformation $g^* \circ v \Rightarrow u\circ f^*$ is an
equivalence. We will write
$$\mcat \subset \Fun(\Delta^1,\scat)$$
for the (non-full) sub-$(2,\s)$-category of $\mu$-categories and $\mu$-morphisms.
\end{df}
\begin{rmk}
Let us highlight the variance of the $\s$-functors in the above definition: we have
decided to use the \emph{left} adjoints $u$ and $v$, instead of their right adjoints
$u_*$ and $v_*$ to specify the variance. A $\mu$-morphism is therefore given by a commutative square as in Definition \ref{d1-2} that is \emph{right adjointable} in the sense of \cite[Def.~4.7.4.13]{HA}.
\end{rmk}
\begin{rmk}\label{r-1}
The condition on a $\mu$-morphism implies that we have several canonical 
equivalences between functors, obtained by taking various adjoints
$$g^*\circ v \simeq u\circ f^* \quad\qquad v\circ f_* \simeq g_* \circ u \quad\qquad v_*\circ g_* \simeq f_*\circ u_*.$$  
Using this, there are various equivalent ways of defining $\mu$-morphisms. For instance, a $\mu$-morphism from $T_1' \to T_1$ to $T_2' \to T_2$ can equivalently be encoded by a \emph{left adjointable} square in $\scat$ of the form
$$\xymatrix{
T_2' \ar[r]^{u_*}  \ar[d]_-{f_*} & T_1' \ar[d]^-{g_*} \\
T_2 \ar[r]_{v_*} & T_1,
}$$
This exhibits $\mcat\subset \Fun(\Delta^1,\scat^{R}_{pr})^{op}$ as a (non-full) sub-$(2, \s)$-category of the \emph{opposite} $(2, \s)$-category of morphisms in $\scat^R_{pr}$.
\end{rmk}

By definition, an object $T' \to T$ can also be written, up to equivalence, as
the forgetful $\s$-functor $\LMod_M(T) \to T$ for a monad $M$ on $T$, which furthermore
commutes with sifted colimits. Objects in $\mcat$ can therefore be identified, up to equivalence, with pairs $(T', M)$, with $T$ a presentable $\s$-category and $M$ a sifted colimit preserving 
monad on $T$. Morphisms
in $\mcat$ can then be viewed as 
$\s$-functors endowed with compatibility data with the monads; see \cite{MR4367222} for more details. This is our justification
for the terminology \emph{$\mu$-category}, where $\mu$ stands for \emph{monad}.

We will often present objects in $\mcat$ as pairs $(T, M)$ of a presentable $\s$-category $T$ 
and a sifted colimit preserving monad $M$ on $T$. The $\s$-category $T$ will be called the 
\emph{underlying $\s$-category} of the object $(T, M)$. The association $(T, M) \mapsto T$
defines an evident forgetful $(2,\s)$-functor
$$\mcat  \to \scat^L_{pr},$$
where $\scat_{pr}^L$ is the $(2,\s)$-category of presentable $\s$-categories and left adjoints as morphisms. 
The $(2,\s)$-functor above is also induced by the composition of the inclusion $\mcat \subset \Fun(\Delta^1,\scat)$ and the evaluation $\Fun(\Delta^1,\scat)\to \scat$ at the object $1 \in \Delta^1$.
The forgetful $(2,\s)$-functor $\mcat \to \scat_{pr}^L$ is easily seen to be conservative. Indeed, if 
$$\xymatrix{
T_1' \ar[r]^{u}  \ar[d]_-{f_*} & T_2' \ar[d]^-{g_*} \\
T_1 \ar[r]_{v} & T_2,
}$$
is a right adjointable square in $\scat$ and if $v$ is an equivalence, then $u$ is automatically 
an equivalence. The adjointability property implies that the unit and counit 
of the adjunction $(u,u_*)$
$$a \colon u \circ u_* \Rightarrow id \qquad b : id \Rightarrow u_* \circ u$$
are sent by $f_*$ and $g_*$ to the unit and counit of the adjunction $(v,v_*)$. 
As $v$ is an equivalence, these are equivalences in $T_1$ and $T_2$, and as $f_*$ and $g_*$
are conservative, $a$ and $b$ must be equivalences to. \\

We finish this first section by noticing that $\mcat$ receives a natural $(2,\s)$-functor
from the $(2,\s)$-category of presentable symmetric monoidal $\s$-categories. For this, we let 
$\scat^{\otimes}_{pr}$ be the $(2,\s)$-category of presentable symmetric monoidal $\s$-categories, 
and colimit preserving symmetric monoidal $\s$-functors as morphisms.
To each $T \in \scat^\otimes_{pr}$ we can associate the $\s$-category of commutative algebras in $T$, 
formally defined as $\CAlg(T) := \Fun^{lax-\otimes}(*,T)$, the $\s$-category of symmetric lax-monoidal 
$\s$-functors from the punctual $\s$-category (with its canonical symmetric monoidal
structure) to $T$. Evaluation at the unique object $*$ provides a well-defined
right adjoint $\s$-functor $\CAlg(T) \to T$. This construction extends to
a $(2,\s)$-functor
$$\scat^{\otimes}_{pr} \to \mcat; \quad T \mapsto (\CAlg(T) \to T)$$
from symmetric monoidal $\s$-categories to $\mu$-categories. To see this, note that each symmetric monoidal left adjoint $\s$-functor $u\colon T_1\to T_2$ defines a left adjoint in the $(2,\s)$-category of symmetric monoidal $\s$-categories and symmetric lax-monoidal $\s$-functors (see \cite[Cor.~7.3.2.7]{HA}). 
Composition with $(u, u_*)$ then defines an adjunction between $\s$-categories of commutative algebras, which commutes with the forgetful functors; this implies that $u$ induces a natural transformation from $\CAlg(T_1)\to T_1$ to $\CAlg(T_2)\to T_2$ that is right adjointable.

In terms of monads, one can think of $\scat^{\otimes}_{pr} \to \mcat$ as sending
a presentable symmetric monoidal $\s$-category $T$ to the pair $(T,M)$, consisting of the
underlying $\s$-category $T$ endowed with the $E_\s$-algebra monad 
$$\Sym_{E_\s}(x) = \coprod_{n\in \NN}(x^{\otimes n})_{h\Sigma_n}.$$

\subsection{$\Theta$-Categories}

Let $C$ be the $(2,\s)$-category defined by the following lax-cartesian square
$$\xymatrix{
\scat^{\otimes}_{pr} \ar[r]^-{q} \ar@{}[rd]^(0.3){}="s"^(0.7){}="t"\ar@{=>} "s";"t" & \Fun(\Delta^1, \scat_{pr}^{R})^{op} \\
C \ar[r] \ar[u] & \mcat. \ar[u]_-{p}
}$$
In this diagram, the top right corner is the $(2,\s)$-category of
all morphisms in $\scat_{pr}^R$, the $(2,\s)$-category of 
presentable $\s$-categories and right adjoints.
The top horizontal arrow $q$ arises from the canonical construction discussed at the end of the last section: 
it sends 
$T \in \scat_{pr}^\otimes$ to the pair $\CAlg(T) \to T$, and a symmetric monoidal $\s$-functor $u \colon T_1 \to T_2$ to its right adjoint 
$$\xymatrix{
\CAlg(T_2) \ar[r]^-{u_*} \ar[d] & \CAlg(T_1) \ar[d] \\
T_2 \ar[r]_-{u_*} & T_1.}$$
The morphism $p$ is the inclusion from Remark \ref{r-1}, sending a $\mu$-category $(T'\to T)$ to $T' \to T$, and a morphism $(u,v)$ to its
right adjoint square $(u_*, v_*)$.

Objects in $C$ can be thought of as triples $(T_1, (T_2'\to T_2), u)$, where $T_1$ is a presentable
symmetric monoidal $\s$-category, $(T_2' \to T_2)$ is a $\mu$-category and 
$u$ is a commutative diagram in $\scat_{pr}^{R}$ 
$$\xymatrix{
T_2' \ar[r] \ar[d] & \CAlg(T_1) \ar[d]  \\
T_2 \ar[r] &T_1 .
}
$$

\begin{df}\label{d1-3}
The \emph{$(2,\s)$-category of $\Theta$-categories} is the full sub-$(2,\s)$-category 
of the $(2, \s)$-category $C$ defined above consisting of the objects $(T_1,(T_2'\to T_2),u)$ that satisfy the following conditions.
\begin{enumerate}
    \item The induced (right adjoint) $\s$-functor $T_2 \to T_1$ is an equivalence.
    \item The induced (right adjoint) $\s$-functor $T_2'\to \CAlg(T_1)$
    commutes with arbitrary colimits.
\end{enumerate}
This $(2,\s)$-category is denoted by $\tcat$.
\end{df}

In more intuitive terms, objects in $\tcat$ are presentable symmetric monoidal $\s$-categories 
$T$ together with a monad $M$ on their underlying $\s$-category, as well as a
morphism of sifted colimit preserving monads $Sym_{E_\s} \to M$
satisfying condition $(2)$ from Definition \ref{d1-3}. Note that, because push-outs in 
$\CAlg(T)$ are given by tensor products, this condition has the following equivalent reformulations.
\begin{enumerate}
\item[(2')] For $x, y\in T$, the natural maps of $E_\infty$-algebras $\mathbf{1}_M\to M(\emptyset)$ and $M(x)\otimes M(y) \to M(x\amalg y)$ are equivalences.

\item[(2'')] The forgetful functor $\LMod_M(T)\to \CAlg(T)$ preserves push-outs, that is, push-outs in $\LMod_M(T)$ are compatible with tensor products in $T$.
\end{enumerate}
The equivalence of these conditions with condition (2) from Definition \ref{d1-3}
follows from the fact $\LMod_M(T)\to \CAlg(T)$ already preserves sifted colimits, so that it preserves all colimits if and only if it preserves finite coproducts of free algebras.

There is an obvious forgetful $(2,\s)$-functor
$$\tcat \to \scat^\otimes_{pr},$$
which forgets the monad $M$. Note that it is conservative, because its composition with the conservative
$\s$-functor $\scat^{\otimes}_{pr}\to \scat^{L}_{pr}$ forgetting the monoidal structure is equivalent to the
forgetful $\s$-functor sending a triple $(T_1, (T_2'\to T_2), u)$ to $T_2 \simeq T_1$.

On the other hand, any
symmetric monoidal $\s$-category $T$ can be sent to the object in $\tcat$ defined by taking $M=\Sym_{E_\s}$, 
or equivalently to the triple $(T, (\CAlg(T)\to T), id)$.
By construction, it is a fully faithful $(2,\s)$-functor and provides a full embedding
of $(2,\s)$-categories
$$\scat_{pr}^{\otimes} \longrightarrow \tcat.$$

\begin{rmk}\label{r-1-3}
The $(1,\s)$-category underlying $\tcat$ has small limits, 
and the forgetful functor $\tcat \to \scat$ preserves these. 
Indeed, the lax pull-back $C$ has limits, computed at the level of the underlying $\s$-categories, 
because the functors $q\colon \scat^{\otimes}_{pr}\to \Fun(\Delta^1, \scat_{pr}^{R})^{op}$ and
$p\colon \mcat \to \Fun(\Delta^1, \scat_{pr}^{R})^{op}$ both preserve limits.
For $q$, this uses the fact that $(T, \otimes)\mapsto \CAlg(T)$ preserves limits by \cite[Prop~2.2.4.9]{HA}. Since the conditions from Definition \ref{d1-2} are stable under limits, $\tcat\subset C$ is closed under limits.
\end{rmk}
\begin{nota}
A $\Theta$-category will be denoted symbolically by 
$T^\Theta$, and its underlying presentable symmetric monoidal
$\s$-category will be denoted by $T$. The \emph{$\Theta$-structure} of $T^\Theta$ is by definition
the extra right adjoint $T'\simeq \LMod_M(T) \to \CAlg(T) \to T$ participating in the data
of a $\Theta$-category. It will be symbolically denoted by 
$$\TCAlg(T^\Theta):=T' \to \CAlg(T)$$
and objects in $\TCAlg(T^\Theta)$ will be called \emph{$\Theta$-algebras}.
\end{nota}
The $\s$-category 
of $\Theta$-algebras comes equipped with a natural adjunction of presentable
$\s$-categories
$$F^\Theta_{E_\s} \colon \CAlg(T) \leftrightarrows \TCAlg(T^\Theta) : U_{E_\s}^{\Theta}.$$
The right adjoint $U_{E_\s}^{\Theta}$ is conservative and, more importantly, also a left adjoint, as it commutes with arbitrary colimits by assumption. 
In particular, it preserves
the initial object. The initial object in $\CAlg(T)$ is the monoidal unit $\mathbf{1} \in T$, which therefore
possesses a canonical compatible $\Theta$-algebra structure. The corresponding object
will still be denoted by $\mathbf{1} \in \TCAlg(T^\Theta)$. \\

We finish this part with a basic description of the $\s$-categories of 
morphisms in $\tcat$. Let $T_1^\Theta$ and $T_2^\Theta$ be two $\Theta$-categories
and $T_1$ and $T_2$ their underlying symmetric monoidal $\s$-categories. Recall that $\tcat$ is defined as a full sub-$(2,\s)$-category of $C$, the lax-fiber product below 
$$\xymatrix{
\scat^{\otimes}_{pr} \ar[r]^-{q} \ar@{}[rd]^(0.3){}="s"^(0.7){}="t"\ar@{=>} "s";"t" & \Fun(\Delta^1, \scat_{pr}^{R})^{op} \\
C \ar[r] \ar[u] & \mcat. \ar[u]_-{p}
}$$
Consequently, the $\s$-category $\Fun_{\tcat}(T_1^{\Theta}, T_2^{\Theta})$ of morphisms in $\tcat$ can be identified with a natural fiber product as well. More precisely, the mapping $\s$-category fits into a square of $\s$-categories
$$\xymatrix{
\Fun_{\tcat}(T_1^{\Theta}, T_2^{\Theta})\ar[r]\ar[d] & \Fun^R\big(\TCAlg(T_2^{\Theta}), \TCAlg(T_1^{\Theta})\big)\ar[d]\\
\Fun_{\scat^{\otimes}_{pr}}(T_1,T_2)\ar[r] & \Fun^R\big(\TCAlg(T_2^{\Theta}), \CAlg(T_1)\big).
}$$
Here the bottom functor sends a left adjoint symmetric monoidal functor $f^*\colon T_1\to T_2$ 
to the composite right adjoint 
$$
f_*\circ U^\Theta_{E_\infty}\colon \TCAlg(T_2^\Theta)\to \CAlg(T_2)\to \CAlg(T_1).
$$
The right functor sends a right adjoint $u_*\colon \TCAlg(T_2^\Theta)\to \TCAlg(T_1^\Theta)$
to the composite $U^{\Theta}_{E_\infty}\circ u_*$. 

\begin{lem}\label{l-3}
The above square exhibits $\Fun_{\tcat}(T_1^{\Theta}, T_2^{\Theta})$ as the full sub-$\s$-category 
of the fiber product on those tuples of a morphism $f^*\colon T_1\to T_2$ in $\scat_{pr}^{\otimes}$ and a right adjoint $u_*\colon \TCAlg(T_2^\Theta)\to \TCAlg(T_1^\Theta)$ 
such that the commutative square in $\scat_{pr}^R$
$$\xymatrix{
\TCAlg(T_2) \ar[d]_-{U^\Theta_{E_\infty}} \ar@{.>}[r]^-{u_*} & \TCAlg(T_1) \ar[d]^{U^\Theta_{E_\infty}} \\
\CAlg(T_2) \ar[r]_-{f_*} & \CAlg(T_1)
}$$
is left adjointable: the induced natural transformation $f^* \circ U^\Theta_{E_\infty} \Rightarrow U^\Theta_{E_\infty}\circ u^*$ is an equivalence.
\end{lem}
\begin{proof}
Let us write $U^{E_\infty}\colon \CAlg(T)\to T$ and 
$U^{\Theta}=U^{E_\infty}\circ U^{\Theta}_{E_\infty}$ for the forgetful functors. 
Unravelling the definitions (using Remark \ref{r-1}), one sees that $\Fun_{\tcat}(T_1^{\Theta}, T_2^{\Theta})$ is 
the full sub-$\infty$-category of the fiber product on the pairs $(f^*, u_*)$ such that
the natural transformation $f^*\circ U^\Theta\Rightarrow U^\Theta\circ u^*$ is an equivalence.
This natural transformation is equivalent to the composition
$f^*\circ U^{E_\infty}\circ U^{\Theta}_{E_\infty} \Rightarrow U_{E_\infty}\circ f^*\circ U^{\Theta}_{E_\infty}\Rightarrow U_{E_\infty}\circ U^{\Theta}_{E_\infty}\circ u^*$. The result follows because the first map is an equivalence and $U^{E_\infty}$ is conservative.
\end{proof}
Taking the fiber of $\Fun_{\tcat}(T_1^\Theta,T_2^\Theta) \to \Fun_{\scat^{\otimes}_{pr}}(T_1,T_2)$ over a fixed morphism $f \colon T_1 \to T_2$ in $\scat_{pr}^\otimes$, we obtain the following proposition.

\begin{prop}\label{p1-1}
For two $\Theta$-categories $T_1^\Theta$ and $T_2^\Theta$, and 
$f \colon T_1 \to T_2$ in $\scat_{pr}^{\otimes}$ a fixed symmetric monoidal $\s$-functor, 
we have a cartesian square
$$ \xymatrix{\ar[d] \Fun^{R, adj}_{/\CAlg(T_1)}(\TCAlg(T_2),\TCAlg(T_1)) 
\ar[r] & 
\Fun_{\tcat}(T_1^\Theta,T_2^\Theta) \ar[d] \\
\{*\} \ar[r]_{f} & \Fun_{\scat_{pr}^{\otimes}}(T_1,T_2),
}$$
where 
$$\Fun^{R, adj}_{/\CAlg(T_1)}(\TCAlg(T_2),\TCAlg(T_1))  \subset \Fun^{R}_{/\CAlg(T_1)}(\TCAlg(T_2),\TCAlg(T_1)) $$ 
is the full sub-simplicial set of right adjoint functors $u_*$ 
making the square from Lemma \ref{l-3} commute and left adjointable.
\end{prop}
\begin{rmk}\label{r-1-2}
The $\s$-category $\Fun^R_{/\CAlg(T_1)}(\TCAlg(T_2),\TCAlg(T_1))$, and hence the fiber of $f$, is automatically an $\s$-groupoid because the forgetful $\s$-functor $\TCAlg(T_1) \to \CAlg(T_1)$
is conservative. 
\end{rmk}

\subsection{$\Theta$-Categories of modules}

To conclude this part on $\Theta$-categories, we will define a $\Theta$-category of left $B$-modules for a given $\Theta$-algebra $B \in \TCAlg(T^\Theta)$. To this end, we first consider
the underlying commutative algebra $B_0:=U_{E_\s}^{\Theta}(B)$ and the
corresponding symmetric monoidal $\s$-category $\LMod_{B_0}(T)$ of left $B_0$-modules
in $T$ endowed with its natural symmetric monoidal structure $\otimes_{B_0}$. It will 
be denoted by $B_0\text{-}\Mod(T)$. 

We define a $\Theta$-structure on $B_0\text{-}\Mod(T)$ by considering the composition of forgetful $\s$-functors
$$B/\TCAlg(T^\Theta) \longrightarrow B_0/\CAlg(T)\simeq \CAlg(B_0\text{-}\Mod(T)) \longrightarrow B_0\text{-}\Mod(T).$$
The above sequence of $\s$-functors defines an new $\Theta$-category whose underlying 
$\s$-category is the $\s$-category of $B_0$-modules and whose $\s$-category of $\Theta$-algebras
is $B/\TCAlg(T^\Theta)$, the $\s$-category of $\Theta$-algebras under $B$.

\begin{df}\label{d1-5}
We will refer to the $\Theta$-category constructed above 
$$B/\TCAlg(T^\Theta) \longrightarrow \CAlg(B_0\text{-}\Mod(T)) \longrightarrow B_0\text{-}\Mod(T)$$
as the \emph{$\Theta$-category of $B$-modules in $T$}, and denote it simply by 
$B\text{-}\Mod(T^\Theta)$.
\end{df}

The $\Theta$-category of $\Theta$-modules in $T^\Theta$ receives a natural morphism from $T^\Theta$
by tensoring with $B$. This natural morphism $T^\Theta \to B\text{-}\Mod(T^\Theta)$
can be represented by the following commutative diagram of $\s$-categories
$$\xymatrix{
\TCAlg(T^\Theta) \ar[r] \ar[d]_-{\coprod B} & \CAlg(T) \ar[d]^-{\otimes B_0} \ar[r] & 
T \ar[d]^-{\otimes B_0} \\
B/\TCAlg(T^\Theta)  \ar[r] & B_0/\CAlg(T) \ar[r] & B_0\text{-}\Mod(T).
}$$
Here the rows are right adjoints whereas the vertical morphisms are left adjoints, whose right adjoints (forgetting the maps from $B$ or $B_0$) commute with the horizontal functors.

We can now state the following proposition, which is an extension to the context of $\Theta$-categories
of the results \cite[Cor.~4.8.5.21]{HA} on symmetric monoidal $\s$-categories and commutative algebras.

\begin{prop}\label{p1-2}
Let $T^\Theta$ be a $\Theta$-category. The $\s$-functor sending a $\Theta$-algebra $B \in \TCAlg(T^\Theta)$
to $T^\Theta \to B\text{-}\Mod(T^\Theta)$ induces a 
fully faithful left adjoint $\s$-functor
$$\TCAlg(T^\Theta) \longrightarrow T^\Theta/\tcat.$$
\end{prop}

\begin{proof}
Unwinding the definitions, this is a consequence of the following two statements.
\begin{enumerate}
    \item The $\s$-functor sending a commutative algebra $B_0 \in \CAlg(T)$
to $T \to B_0\text{-}\Mod(T)$ induces a fully faithful left adjoint functor of $(2, \s)$-categories
$$\CAlg(T) \longrightarrow T/\scat_{pr}^{\otimes}.$$
    \item For any presentable $\s$-category $C$, the $\s$-functor
    $$C \to (\scat_{pr}^R/C)^{op}$$
    sending $x \in C$ to the comma $\s$-category $x/C$ with the projection $x/C \to C$, 
    and a morphism $u \colon y\to x$ to the composition $x/C \to y/C$ with $u$,
    defines a full embedding.
\end{enumerate}

Note that the $\infty$-category of maps $B\text{-}\Mod(T)\to T'$ in $T/\scat^\otimes_{pr}$ 
is in fact a space: 
any natural transformation $\mu\colon f\Rightarrow g$ between morphisms 
$f, g\colon B\text{-}\Mod(T)\to T'$ in $T/\scat_{pr}^{\otimes}$ is an equivalence.
Indeed, the class $B$-modules $M$ such that $\mu_M$ is an equivalence
is closed under colimits and contains the free $B$-modules.
In light of this, statement $(1)$ is proven in \cite[Thm.~4.8.5.11, Cor.~4.8.5.21]{HA}, 
whereas statement $(2)$ is simply a version of the Yoneda lemma. \\

Let $B\in \TCAlg(T^{\Theta})$ be a $\Theta$-algebra in $T^\Theta$ 
and $g\colon T^{\Theta}\to T_2^{\Theta}$ a $\Theta$-functor. 
Let $g_*(\mathbf{1})\in \TCAlg(T^{\Theta})$ be the image of $\mathbf{1}\in \TCAlg(T_2^{\Theta})$ 
under the right adjoint $g_*$. 
We then have the following commutative square of mapping spaces (rather than $\s$-categories, by Remark \ref{r-1-2})
$$\xymatrix{
\Map_{T^\Theta/\tcat}(B\text{-}\Mod(T^\Theta),T_2^{\Theta}) \ar[r] \ar[d] &
\Map_{\TCAlg(T^\Theta)}(B,g_*(\mathbf{1})) \ar[d] \\
\Map_{T/\scat_{pr}^\otimes}(B_0\text{-}\Mod(T),T_2) \ar[r] & \Map_{\CAlg(T)}(B_0,g_*(\mathbf{1})).}
$$
Here the top horizontal map sends a $\Theta$-functor 
$f\colon B\text{-}\Mod(T^{\Theta})\to T_2^{\Theta}$ to the unit map 
$B\to f_*(\mathbf{1})$ in $\TCAlg(B\text{-}\Mod(T^{\Theta}))$. 
This corresponds to a map $B\to g_*(\mathbf{1})$ in $\TCAlg(T^{\Theta})$ 
under the equivalence 
$$
\TCAlg(B\text{-}\Mod(T^{\Theta}))\simeq B/\TCAlg(T^{\Theta}).$$
The bottom horizontal map is defined similarly.

It now suffices to verify that the top horizontal map is an equivalence. 
Indeed, this shows that the functor $B\mapsto B\text{-}\Mod(T^{\Theta})$
 is a left adjoint. Taking $T_2^{\Theta}=B_2\text{-}\Mod(T^{\Theta})$, 
so that $g_*(\mathbf{1})=B_2$, shows that it is fully faithful.

By property (1) above, the bottom horizontal map is an equivalence. 
Let us therefore fix a symmetric monoidal $\s$-functor 
$f\colon B\text{-}\Mod(T)\to T_2$ under $T$, 
and compare the induced map between the fibers of the two vertical maps. 
By Proposition \ref{p1-1}, there is a canonical equivalence
between $\Map_{T^\Theta/\tcat}(B\text{-}\Mod(T^\Theta),T_2^\Theta)_f$ and the 
space of all possible limit preserving $\s$-functors $u$ filling up the diagram below
in $\scat^R_{pr}$
$$\xymatrix{
\ar@{.>}@/_4pc/[dd]_ -{u} \TCAlg(T_2^\Theta) \ar[d]_-{g_*} \ar[r] & \ar@/_3pc/[dd]_>>>>>>>{f_*} 
\CAlg(T_2) \ar[d]^-{g_*} \\
\TCAlg(T^\Theta) \ar[r]|(0.635)\hole & \CAlg(T) \\
B/\TCAlg(T^\Theta) \ar[r] \ar[u] & B_0/\CAlg(T) \ar[u] 
}$$
such that $f^*\circ U^{\Theta}_{E_\infty}\Rightarrow U^{\Theta}_{E_\infty}\circ u^*$ is an equivalence. 

Note that there is an equivalence of $\s$-categories
$$
\TCAlg(T_2^{\Theta})\times_{\TCAlg(T^{\Theta})} B/\TCAlg(T^{\Theta}) \simeq g(B)/\TCAlg(T_2^{\Theta})
$$
since $g\colon \TCAlg(T^\Theta)\to \TCAlg(T_2^{\Theta})$ is left adjoint to $g_*$, and similarly for commutative algebras. Property (2) recalled at the start of the proof therefore identifies the space of $u$ filling the diagram with the fiber at $f$ of the natural morphism
$$\Map_{\TCAlg(T_2^\Theta)}(g(B),\mathbf{1}) \to \Map_{\CAlg(T_2)}(g(B_0),\mathbf{1}).$$
By adjunction, this is equivalent to the fiber at $f$ of the natural morphism
$$\Map_{\TCAlg(T^\Theta)}(B,g_*(\mathbf{1})) \to \Map_{\CAlg(T)}(B_0,g_*(\mathbf{1})).$$
The result now follows by noting that for each map of $\Theta$-algebras $B\to g_*(\mathbf{1})$, the natural transformation $g_*(\mathbf{1})\otimes_{B_0} U^{\Theta}_{E_\infty}(-)\Rightarrow U^{\Theta}_{E_\infty}(g_*(\mathbf{1})\amalg_B -)$ is an equivalence since $U^{\Theta}_{E_\infty}$ preserves colimits.
\end{proof}

Proposition \ref{p1-2} can be usefully combined with the following recognition result.
For this, we note that a $\Theta$-functor $f \colon T_1^\Theta \to T_2^\Theta$ always induces
an adjunction on the level of the corresponding $\s$-categories of $\Theta$-algebras
$$f \colon \TCAlg(T_1^\Theta) \leftrightarrows \TCAlg(T_2^\Theta) : f_*$$
where $f$ and $f_*$ commute
with the forgetful $\s$-functors to $T_1$ and $T_2$. 
For any $B \in \TCAlg(T_1^\Theta)$, this yields
a well-defined $\Theta$-functor at the level of modules
$$f \colon B\text{-}\Mod(T_1^\Theta) \to f(B)\text{-}\Mod(T_2^\Theta).$$
Applying this to $B = f_*(\mathbf{1})$, and composing with the base change along the counit
$f(f_*(\mathbf{1})) \to \mathbf{1}$, we obtain a natural $\Theta$-functor 
$f_*(\mathbf{1})\text{-}\Mod(T_1^\Theta) \to T_2^\Theta$
whose underlying adjunction of $\s$-categories takes the form
$$
\mathbf{1}\otimes_{f(f_*(\mathbf{1}))} f(-)\colon f_*(\mathbf{1})\text{-}\Mod(T_1)\leftrightarrows T_2 : f_*.
$$
\begin{prop}\label{p1-3}
Let $T_1^\Theta \to T_2^\Theta$ be a $\Theta$-functor and $f \colon T_1 \to T_2$ the underlying 
symmetric monoidal $\s$-functor. We assume that the right adjoint $f_* \colon T_2 \to T_1$ of $f$ satisfies the following conditions.
\begin{enumerate}
    \item The $\s$-functor $f_*$ is conservative and preserves geometric realizations.
    \item For $y \in T_2$ and $x \in T_1$, the natural morphism
    $$x\otimes f_*(y) \to f_*(f(x) \otimes y)$$
    is an equivalence in $T_1$.
\end{enumerate}
Then the natural $\Theta$-functor
$$f_*(\mathbf{1})\text{-}\Mod(T_1^\Theta) \longrightarrow T_2^\Theta$$
is an equivalence of $\Theta$-categories.
\end{prop}
\begin{proof}
The underlying symmetric monoidal $\s$-functor 
$f_*(\mathbf{1})\text{-}\Mod(T_1) \to T_2$ is an equivalence by the second part of \cite[Cor.~4.8.5.21]{HA}. As forgetting the $\Theta$-structure $\tcat \to \scat_{pr}^{\otimes}$
is conservative, this implies the proposition.
\end{proof}

\section{Tannaka duality}

In this section, we fix a commutative ground ring $k$. We denote by $\St_k$ the 
$\s$-category of stacks on the big site of affine $k$-schemes endowed with the 
fpqc topology.

\subsection{$\Theta$-Categories of quasi-coherent and ind-perfect sheaves}

For a commutative $k$-algebra $A$, we consider $\QCoh(\Spec\, A)$, also denoted by $\QCoh(A)$,
the symmetric monoidal $\s$-category of complexes of $A$-modules. When $A$ varies in 
commutative $k$-algebras this defines an fpqc-stack in presentable symmetric monoidal $\s$-categories
$$\QCoh \colon \Aff_k^{op} \longrightarrow \scat_{pr}^\otimes.$$
We will need a slight modification of the stack $\QCoh$, which involves
ind-perfect complexes. 
For any small stack $F$, the dualizable objects in the symmetric monoidal $\s$-category $\QCoh(F)$ are the perfect complexes, 
whose $\s$-category will be denoted $\Parf(F)$.
It comes equipped with the restriction of the symmetric monoidal
structure of $\QCoh(F)$, and as such forms a small, stable symmetric monoidal $\s$-category. 

\begin{df}
The \emph{presentable symmetric monoidal $\s$-category of ind-perfect complexes on $F$}
is the ind-completion of $\Parf(F)$. It is denoted by 
$$\IParf(F):=\Ind(\Parf(F)).$$
\end{df}

For any small stack $F \in \St_k$ we have a natural adjunction of presentable symmetric monoidal $\s$-categories
$$\IParf(F) \leftrightarrows \QCoh(F),$$
where the left adjoint is the canonical extension by filtered colimits of the inclusion $\Parf(F)\subset \QCoh(F)$.
When $X=\Spec\, A$ is an affine $k$-scheme, or more generally a quasi-compact and quasi-separated $k$-scheme, then the perfect complexes are compact generators
of $\QCoh(X)$, and thus the above adjunction
is an equivalence. However, we warn the reader that for a general stack, even representable by a nice
Artin stack, perfect complexes might not even be compact objects in quasi-coherent complexes. In general, the natural $\s$-functor $\IParf(F) \to \QCoh(F)$ is therefore
neither fully faithful nor essentially surjective. \\

We want to enhance the two $\s$-functors $F \mapsto \IParf(F)$ and $F \mapsto \QCoh(F)$ from 
small stacks to $\scat_{pr}^\otimes$, to $\s$-functors to $\tcat$. 
This is straightforward for quasi-coherent complexes (see Remark \ref{r-2-1} below), 
but more involved for ind-perfect complexes because the (L)Sym-monad does not preserve perfect complexes. 
We will therefore
construct explicit models for these symmetric monoidal $\s$-categories with strict representatives
of the involved monads. A more intrinsic approach should follow from 
the techniques of \cite[\S 4.2]{raksit2020}. 

We let $X_*\colon I \to \Aff_k$ be a small diagram of affine $k$-schemes over $F$ such that 
the canonical morphism $\colim_{i\in I}X_i \to F$ is an equivalence. Dually, this yields
a diagram of (discrete) commutative $k$-algebras $A_* \colon I^{op} \to k\text{-}\CAlg^{\heart}$. This diagram can be
seen as a commutative algebra object in the ((1, 1)-) category of presheaves of abelian groups
$\Fun(I^{op},\Ab)$, so that we can consider the abelian category $A_*\text{-}\Mod$ of presheaves of $A_*$-modules.
Recall that an object $M$ in $A_*\text{-}\Mod$ consists of $A_i$-modules $M_i$ for all $i \in I$, 
and morphisms $u^* \colon A_j \otimes_{A_i}M_i \to M_j$ of $A_j$-modules for
all $u \colon j \to i$ in $I$, satisfying the
usual compatibilities with compositions and identity in $I$. The abelian category $A_*\text{-}\Mod$ possesses
enough projective objects: these are the direct sums of free $A_*$-modules $F(i)$ on objects $i \in I$,
characterized by the functorial bijection $\mathit{Hom}(F(i),M) \simeq M_i$. 
They can also be written more explicitly
as $F(i)(j) = \bigoplus_{j \to i}A_j$ for $j\in I$.

We let $scA_*\text{-}\Mod$ be the
category of simplicial-cosimplicial objects in the abelian category $A_*\text{-}\Mod$. As explained in \cite{Joostsc}
the category $scA_*\text{-}\Mod$ is endowed with a model category structure whose weak equivalences are 
induced by the totalization functor to presheaves of cochain complexes of abelian groups
$$Tot^\oplus \colon scA_*\text{-}\Mod \longrightarrow \Fun(I^{op},C(\ZZ)).$$
The fibrations are the levelwise fibrations of simplicial-cosimplicial abelian groups
of \cite[Defn.~5.2]{Joostsc}. A set of generating cofibrations is given by 
the family of morphisms $F(i)\otimes_\ZZ E \to F(i)\otimes_\ZZ E'$ where $i$ varies in $I$ and 
$E \to E'$ is a generating cofibration of the model category of simplicial-cosimplicial abelian groups (see
proof of \cite[Thm.~5.5]{Joostsc}). We note in particular that the evaluation 
functor at $i \in I$, $scA_*\text{-}\Mod \to scA_i\text{-}\Mod$ preserves cofibrations and trivial cofibrations.

Let $L(scA_*\text{-}\Mod)$ be the $\s$-category obtained from $scA_*\text{-}\Mod$ 
by inverting the weak equivalences. 
The $\s$-category $L(scA_*\text{-}\Mod)$ is presentable and closed symmetric monoidal, and is equivalent to the lax limit of the diagram
$i \mapsto \QCoh(A_i)$ (see e.g. \cite{Harpaz}). 
The limit of this diagram, $\QCoh(F)$, 
sits as a full symmetric monoidal sub-$\s$-category in $L(scA_*\text{-}\Mod)$: it consists precisely 
of objects $M$ such that for any morphism $j \to i$ in $I$ the induced morphism
$$A_j\otimes^{\mathbb{L}}_{A_i}M_i \to M_j$$
is an equivalence of simplicial-cosimplicial modules (i.e. a quasi-isomorphism on the corresponding $Tot^\oplus$).

Let us now consider the category $scA_*\text{-}\CAlg$ of strictly commutative simplicial-cosimplicial $A_*$-algebras,
together with its natural forgetful functor $scA_*\text{-}\CAlg \to scA_*\text{-}\Mod$.
We can lift the (semi-)model category structure along this forgetful functor, which now becomes
a right Quillen functor preserving all weak equivalences. 
It induces a right adjoint $\s$-functor between
presentable $\s$-categories
$$L(scA_*\text{-}\CAlg) \to L(scA_*\text{-}\Mod).$$
Similarly, if $\OO$ is a cofibrant model for the simplicial $E_\s$-operad we can 
consider $\OO$-algebras inside $scA_*\text{-}\Mod$, together with its natural (semi-)model category structure
lifted from the forgetful functor $scA_*\text{-}\Alg^\OO \to scA_*\text{-}\Mod$. The canonical projection
$\OO \to *$ to the terminal operad provides a factorization of (right Quillen) forgetful functors
$$scA_*\text{-}\CAlg \to scA_*\text{-}\Alg^\OO \to scA_*\text{-}\Mod.$$
\begin{lem}\label{l-2-1}
Upon inverting weak equivalences, the above sequence of right Quillen functors induces a sequence of right adjoints between presentable $\s$-categories
$$L(scA_*\text{-}\CAlg) \to L(scA_*\text{-}\Alg^\OO)\simeq \CAlg(L(scA_*\text{-}\Mod)) \to L(scA_*\text{-}\Mod).$$
that endows the presentable symmetric monoidal $\infty$-category $L(scA_*\text{-}\Mod)$ with the structure of a $\Theta$-category.
\end{lem}
\begin{proof}
On localizations, we obtain a sequence of conservative right adjoints that furthermore preserve geometric realizations, since these can be computed by the diagonal in the simplicial direction. In particular, $L(scA_*\text{-}\CAlg)$ and $L(scA_*\text{-}\Alg^\OO)$ can be identified with $\s$-categories of modules for two monads on $L(scA_*\text{-}\Mod)$. In the second case, this is the $E_\infty$-monad and we obtain the middle equivalence. Both monads preserve sifted colimits by (the proof of) \cite[Prop.~5.17]{Joostsc}, so that the forgetful $\s$-functors preserve sifted colimits as well. Finally, the forgetful functor $L(scA_*\text{-}\CAlg) \to L(scA_*\text{-}\Alg^\OO)\simeq \CAlg(L(scA_*\text{-}\Mod))$ also preserves coproducts, which are computed by (derived) tensor products.
\end{proof}
Finally, we define a $\Theta$-structure on $\QCoh(F)$ simply by pull-back along the inclusion $\QCoh(F) \subset
L(scA_*\text{-}\Mod)$
$$\xymatrix{
L(scA_*\text{-}\CAlg) \ar[r] &  \CAlg(L(scA_*\text{-}\Mod)) \ar[r] &  L(scA_*\text{-}\Mod) \\
\TCAlg(\QCoh(F)) \ar[r] \ar[u] & \CAlg(\QCoh(F)) \ar[r] \ar[u] & \QCoh(F), \ar[u]
}$$
where each square above is cartesian. Using Lemma \ref{l-2-1} and the fact that $\QCoh(F)\subset L(scA_*\text{-}\Mod)$ is closed under colimits, one sees that this indeed defines a $\Theta$-structure. We will write
$\QCoh^{\LSym}(F)$ for this $\Theta$-structure and refer to it as the \emph{LSym $\Theta$-structure}. We leave it to the reader to construct
functorialities in $F$ to provide a well defined $\s$-functor $F \mapsto \QCoh^{\LSym}(F)$ 
from small stacks to $\Theta$-categories.
\begin{rmk}\label{r-2-1}
One can also give a more intrinsic description of $\QCoh^{\LSym}(F)$, as follows. Let $\LSym\text{-}\CAlg(k)$ denote the $\s$-category of $\LSym$-algebras over $k$, that is, algebras for the $\LSym$-monad on $\QCoh(k)$ induced by the symmetric algebra monad on flat $k$-modules by right-left extension (see \cite{Joostsc,raksit2020}). This comes with a forgetful functor $\LSym\text{-}\CAlg(k)\to k\text{-}\CAlg$ to $E_\infty$-algebras over $\ZZ$ that preserves limits and colimits (sifted colimits and tensor products).

For each discrete $k$-algebra $A$, we then have a natural sequence of right adjoint forgetful functors
$$
A/\LSym\text{-}\CAlg(k)\to A/k\text{-}\CAlg\to A\text{-}\Mod(\QCoh(k))\simeq \QCoh(A)
$$
depending functorially on $A$. Since the first functor preserves colimits and $A/k\text{-}\CAlg\simeq \CAlg(\QCoh(A))$ by \cite[Cor.~3.4.1.7]{HA}, this determines a $\Theta$-structure on $\QCoh(A)$, equivalent to the LSym $\Theta$-structure considered above. Using this, we obtain a functor $\QCoh^{\LSym}\colon \Aff_k^{op}\to \tcat$ satisfying fpqc descent (since $\QCoh$ does). For each cardinal $\kappa$, we then define
$$
\QCoh^{\LSym}\colon \St_k^{\kappa-small, op} \to \tcat
$$
to be the unique extension preserving $\kappa$-small limits (using Remark \ref{r-1-3}).
\end{rmk}

We now turn to the case of ind-perfect complexes. For the $\Theta$-structure on $\IParf(F)$, we will need a slight modification of the previous construction. For this, 
we let $\mathbb{P} \subset scA_*\text{-}\Mod$ be the full sub-category of objects $M$ satisfying the following 
conditions.
\begin{enumerate}
    \item The object $M$ is cofibrant.
    
    \item For any $i \in I$, $Tot^\oplus(M_i)$ defines a compact object in $\QCoh(A_i)$ and moreover
    for any $i \to j$ the induced morphism $A_j \otimes_{A_i}M_i \to M_j$ is an equivalence
    of simplicial-cosimplicial $A_j$-modules.
\end{enumerate}

The localization of $\mathbb{P}$ along its weak equivalences is equivalent, via the canonical inclusion, to $\Parf(F) \subset \QCoh(F) \subset L(scA_*\text{-}\Mod)$.

For any integer $n\geq 0$ we have an endofunctor
$\Theta_n \colon \mathbb{P} \to \mathbb{P}$ sending $M$ to its genuine $n$-th symmetric power $(M^{\otimes n})_{\Sigma_n}$.  This defines an $\NN$-graded object in the category $\Fun(\mathbb{P},\mathbb{P})$, that is,
a functor
$$\Theta_* \colon \NN \to \Fun(\mathbb{P},\mathbb{P}),$$
where $\NN$ is the discrete category of natural numbers. The functor $\Theta_*$ is moreover an
$\NN$-graded monad, or equivalently, $\Theta_* \colon \NN \to \Fun(\mathbb{P},\mathbb{P})$ admits 
a natural structure of a lax monoidal functor, where the monoidal structure on $\NN$ is given by
the multiplication and the monoidal structure on $\Fun(\mathbb{P},\mathbb{P})$ is the composition of endo-functors.
The lax monoidal structure is given by the natural morphisms
$$\Theta_m(\Theta_n(M_*)) = ((M_*^{\otimes n})_{\Sigma_n})^{\otimes m}_{\Sigma_m} \simeq (M_*^{\otimes nm})_{\Sigma_n \ltimes \Sigma^{\times n}_m} \longrightarrow
(M^{\otimes nm}_*)_{\Sigma_{nm}}$$
induced by the canonical map $\Sigma_n \ltimes \Sigma^{\times n}_m \to \Sigma_{nm}$. The graded
monad $\Theta_*$ will be denoted by $\LSym^*$.

Similarly, any cofibrant simplicial operad $\OO$ such that $\OO(k)_{h\Sigma_k}$ has 
the homotopy type of a finite CW complex gives rise to an $\NN$-graded monad
$$\Sym^*_\OO \colon \NN \to \Fun(\mathbb{P},\mathbb{P})$$
by sending $n$ to $(\OO(n)\otimes (-)^{\otimes n})_{h\Sigma_n}$. The finiteness condition 
on $\OO$ implies that each endofunctor $(\OO(n)\otimes (-)^{\otimes n})_{h\Sigma_n}$
preserves the sub-category $\mathbb{P}$. The lax monoidal structure on $\Sym^*_\OO$ is the usual 
one and endows $\Sym^*_\OO$ with the structure of a graded monad over $\mathbb{P}$. Moreover, the canonical 
projection $\OO \to *$ to the terminal simplicial operad (which is not cofibrant anymore), produces
a canonical morphism of graded monads 
$$\Sym^*_\OO \to \LSym^*.$$

We now apply this to cofibrant models for the $E_n$-operads, and obtain a sequence
of morphisms of graded monads on $\mathbb{P}$
\begin{equation}\tag{S} Id = \Sym_{E_0}^*\to \dots \to \Sym_{E_n}^* \to \Sym^*_{E_{n+1}} \to \dots \to \LSym^*.\end{equation}
These monads preserve the weak equivalences on $\mathbb{P}$, so that they induce a sequence of morphisms of graded monads on the localization $\Parf(F)$. We thus obtain a sequence (S) in the $\s$-category of lax monoidal $\s$-functors
$$\NN \to \Fun(\Parf(F),\Parf(F)),$$
from the multiplicative monoid $\NN$ to the monoidal $\s$-category of endo-$\s$-functors
of $\Parf(F)$. Now recall that there is a canonical full monoidal embedding 
$\Fun(\Parf(F),\Parf(F)) \hookrightarrow \Fun(\IParf(F),\IParf(F))$, whose
essential image consists of filtered colimit preserving $\s$-functors that also preserve
compact objects. We can therefore identify the sequence (S) with a sequence 
in the $\s$-category of graded monads on $\IParf(F)$. Taking the sum over $\NN$, we
get a sequence of monoids in $\Fun(\IParf(F),\IParf(F))$
$$Id \to \dots \to \Sym_{E_n} \to \Sym_{E_{n+1}} \to \dots \to \LSym$$
whose colimit provides the desired morphism of monads on $\IParf(F)$
$$\Sym = \colim_n \Sym_{E_n} \to \LSym.$$
Note that $\Sym$ is identified with the $E_\s$-monad, and therefore the $\s$-categories of algebras
and their forgetful $\s$-functors
$$\LMod_{\LSym}(\IParf(F)) \to \LMod_{\Sym}(\IParf(F)) \to \IParf(F)$$
defines a $\Theta$-structure on $\IParf(F)$.

We have thus constructed
a $\Theta$-category, denoted by $\IParf{}^{\LSym}(F)$, with underlying symmetric monoidal $\s$-category 
equivalent to $\IParf(F)$. We leave it to the reader that this can be made functorial in $F$ and provides
a well defined $\s$-functor
$$\IParf{}^{\LSym}(-) : \St^{small}_k \to \tcat$$
from small stacks to $\Theta$-categories. 
The canonical functor $\IParf(F)\to \QCoh(F)$ can be lifted to a morphism 
of $\Theta$-categories $$\IParf{}^{\LSym}(F) \to \QCoh^{\LSym}(F).$$
Note that this comparison map is an equivalence whenever the stack $F\simeq \Spec\; A$ is affine.
\begin{df}
Let $F$ be a small stack over $k$. We define the
\emph{LSym $\Theta$-categories of quasi-coherent} and \emph{ind-perfect complexes}
on $F$ to be the $\Theta$-categories $\QCoh^{\LSym}(F)$ and $\IParf{}^{\LSym}(F)$
constructed above.
\end{df}
For each stack $F\in \St_k$, the $\Theta$-categories $\QCoh^{\LSym}(F)$
and $\IParf{}^{\LSym}(F)$ are equipped with a natural map
from $\IParf{}^{\LSym}(k)\simeq\QCoh^{\LSym}(k)$. 
They thus acquire a natural $k$-linear structure in the following sense.
\begin{df}\label{d-2-1}
We define a \emph{$k$-linear $\Theta$-category} to be a map of $\Theta$-categories
$\IParf{}^{\LSym}(k)\simeq \QCoh^{\LSym}(k)\to T^{\Theta}$. 
We will write
$$
k / \tcat := \IParf{}^{\LSym}(k) / \tcat
$$
for the $(2,\s)$-category of $k$-linear $\Theta$-categories.
\end{df}
\begin{lem}
The underlying $\s$-category $T$ of a $k$-linear $\Theta$-category $T^\Theta$ is a stable presentable
$\s$-category.
\end{lem}

\begin{proof}
The underlying symmetric monoidal $\s$-functor
$u \colon \QCoh(k) \to T$ induces a natural tensored structure of $T$ over $\QCoh(k)$ by the simple formula
$E \otimes x := u(E)\otimes x$,
for $E \in \QCoh(k)$. As $\QCoh(k)$ is stable, this automatically implies the stability of $T$.
\end{proof}

\subsection{Tannakian $\Theta$-categories and Tannakian gerbes}
For the sake of simplicity we will only deal with the neutral case in the present work. 
\begin{df}
A \emph{neutralized $k$-linear Tannakian $\Theta$-category} is a map of $k$-linear $\Theta$-categories
$$\omega\colon  T^\Theta \longrightarrow \IParf{}^{\LSym}(k)\simeq \QCoh^{\LSym}(k)$$
satisfying the following conditions.

\begin{enumerate}
    \item[(T1)] The underlying presentable symmetric monoidal $\s$-category $T$ is "rigid": it
    is compactly generated and compact objects in $T$ coincide with dualizable objects.
    
    \item[(T2)] The underlying $\s$-functor $\omega \colon T \to \IParf(k) \simeq \QCoh(k)$ 
    is conservative.
    
    \item[(T3)] Let $\QCoh(k)^{\leq 0}\subset \QCoh(k)$ denote the full sub-$\s$-category of 
    connective objects. The full sub-$\s$-category 
    $\omega^{-1}(\QCoh(k)^{\leq 0}) \subset T$ then defines the connective part
    of a bounded (on compact objects)
    $t$-structure on $T$ 
    for which $\omega$ becomes $t$-exact.
    
\end{enumerate}
Let us write $k/\tcat/k := \big(k/\tcat\big)\big/\QCoh^{\LSym}(k)$ for 
the $(2,\s)$-category of $k$-linear $\Theta$-categories over $\QCoh^{\LSym}(k)$ and
$$
k\text{-}\Tan_*^\Theta \subseteq k/ \tcat/ k
$$
for the full sub-$(2,\s)$-category on the neutralized $k$-linear Tannakian $\Theta$-categories. Here
the "$*$" refers to the neutralized assumption. As $\omega$ is conservative, the $(2,\s)$-category
$k\text{-}\Tan_*^\Theta$ is in fact a $(1,\s)$-category. 
\end{df}

\begin{df}
A \emph{pointed} (or \emph{neutralized}) \emph{Tannakian gerbe} over $k$ is a pointed stack $F \in \ast/\St_k$ satisfying the following conditions.

\begin{enumerate}
    \item[(G1)] The natural morphism of fpqc-sheaves $* \to \pi_0(F)$ is an isomorphism.
    
    \item[(G2)] The stack $F$ is \emph{pointed Q-local}: for any morphism of 
    pointed and connected stacks $u \colon G_1 \to G_2$
    whose induced $\s$-functor $u^* \colon \QCoh(G_2) \to \QCoh(G_1)$ is an equivalence, 
    the induced morphism in mapping spaces
    $$u^* \colon \Map_{\St_k}(G_2,F) \to \Map_{\St_k}(G_1,F)$$
    is an equivalence.
    
    \item[(G3)] There exists a coconnective $k$-linear $\LSym$-algebra $B$ and an equivalence
    $$\Spec\, B \simeq \Omega_*F$$
    (see \cite{chaff} for the notion of spectrum of coconnective $\LSym$-algebras and the theory 
    of affine stacks).
    Moreover, $B$ is of positive Tor-dimension (also called connectively flat): if
    $E$ is a coconnective object in $\QCoh(k)$, then $E \otimes_k B$ remains coconnective.
    
    \item[(G4)] The $k$-linear $\Theta$-functor $\omega\colon \IParf{}^{\LSym}(F) \to \IParf{}^{\LSym}(k)$, obtained by 
    pull-back along the base point $* \to F$, makes $\IParf{}^{\LSym}(F)$ into a Tannakian
    $\Theta$-category.
    
\end{enumerate}

We will write $k\text{-}\Gerbe_* \subset */\St_k$ for the full sub-$\s$-category
on the pointed Tannakian gerbes over $k$.
\end{df}

Let $T^\Theta$
be a $k$-linear $\Theta$-category. For a commutative $k$-algebra $A$, we set
$$\Fib^{\LSym}(T^\Theta)(A):=\Map_{k/\tcat}(T,\QCoh^{\LSym}(A)).$$
Since $\QCoh^{\LSym}$ defines an fpqc-stack of $k$-linear $\Theta$-categories, 
varying $A$ inside the category of commutative $k$-algebras yields an fpqc-stack
$\Fib^{\LSym}(T^\Theta)$, called the \emph{stack of $\Theta$-fiber functors on $T^\Theta$}. This is part of an adjunction
$$\xymatrix{
\Fib^{\LSym}\colon k/\tcat\ar@<1ex>[r] &  \St^{op}_k : \QCoh^{\LSym}.\ar@<1ex>[l]
}$$
Indeed, since $\QCoh^{\LSym}$ satisfies fpqc descent, there is a natural equivalence for $F \in \St_k$ and $T^\Theta \in k/\tcat$
$$\Map_{k/\tcat}(T^\Theta,\QCoh^{\LSym}(F)) \simeq 
\Map_{\St_k}(F,\Fib^{\LSym}(T^\Theta)).$$
In particular, we have adjunction morphisms
$$\beta_{T^\Theta} \colon T^\Theta \to \QCoh^{\LSym}\big(\Fib^{\LSym}(T^\Theta)\big) \qquad
F \to \Fib^{\LSym}\big(\QCoh^{\LSym}(F)\big).$$
The second morphism, composed with the canonical morphism $\IParf{}^{\LSym}(F) \to \QCoh^{\LSym}(F)$, induces a morphism
$$\alpha_F \colon F \to \Fib^{\LSym}\big(\IParf{}^{\LSym}(F)\big).$$

This extends naturally to the neutralized, or pointed, situation as follows. Let us write
$\Fib_*^{\LSym}(T^\Theta) \subset \Fib^{\LSym}(T^\Theta)$ for the full sub-stack consisting of all fiber functors locally equivalent to 
$\omega \colon T^{\Theta} \to \QCoh^{\LSym}(k)$. This full sub-stack can also be written
$$\Fib_*^{\LSym}(T^\Theta) \simeq B\underline{aut}^{\Theta}(\omega) \subset \Fib^{\LSym}(T^\Theta)$$
where $\underline{aut}^{\Theta}(\omega)$ is the stack of $\Theta$-auto-equivalences of $\omega$. 
The stack $\Fib_*^{\LSym}(T^\Theta)$ is connected (by definition) and canonically pointed at $\omega$.
Note that 
any $\Theta$-endomorphism of $\omega$ is also a self-equivalence, as a consequence of the
rigidity condition (T1).

Similarly, if $x \colon * \to F$ is a pointed connected stack, we can identify $\QCoh^{\LSym}(F)$ 
with an object in $k/\tcat/k$ using its canonical fiber functor at the basepoint
$x^* \colon \QCoh^{\LSym}(F) \to \QCoh^{\LSym}(k)$. 
This yields a similar adjunction
$$\xymatrix{
\Fib_\ast^{\LSym}\colon k/\tcat/k\ar@<1ex>[r] &  (\ast/\St^{cn}_k)^{op} : \QCoh^{\LSym}\ar@<1ex>[l]
}$$
where $\ast/\St^{cn}_k$ denotes the $\s$-category of pointed connected stacks.
Using the unit and counit of this adjunction, together with natural morphism $\IParf{}^{\LSym} \to \QCoh^{\LSym}(F)$, we obtain
two canonical morphisms
$$\beta_{T^\Theta} \colon T^\Theta \to \QCoh^{\LSym}(\Fib_*^{\LSym}(T^\Theta))
 \qquad\quad \alpha_F \colon F \to \Fib_*^{\LSym}(\IParf{}^{\LSym}(F))$$

The Tannakian duality investigated in the next chapter ensures 
that $\beta$ and $\alpha$ are somehow equivalences when $T^\Theta$ is Tannakian and
$F$ is a Tannakian gerbe. We will only have partial answers,
but we refer to Remarks \ref{r1} and \ref{r2} for 
some indications of 
how better results can be obtained.

\subsection{The Tannakian duality theorem}

We start by investigating the morphism $\beta$.

\begin{thm}\label{t1}
Let $T^\Theta$ be a neutralized Tannakian $\Theta$-category over $k$. Then the following hold.
\begin{enumerate}
    \item The unit morphism of $\Theta$-categories 
    $\beta_{T^\Theta}\colon T^\Theta \to \QCoh^{\LSym}(\Fib^{\LSym}_*(T^\Theta))$ induces an 
    equivalence of symmetric monoidal $\s$-categories
    $$T \simeq \IParf(\Fib^{\LSym}_*(T^\Theta)).$$
    \item The stack $\Fib_*^{\LSym}(T^\Theta)$ is a pointed Tannakian gerbe.
\end{enumerate}
\end{thm}
The proof of Theorem \ref{t1} relies on the following property of the fiber functor
of a neutralized Tannakian $\Theta$-category.
\begin{lem}\label{l-2-2}
Let $T^\Theta$ be a neutralized Tannakian $\Theta$-category over $k$. Then there is an equivalence of $\Theta$-categories under $T^{\Theta}$
$$\xymatrix{
& T^{\Theta}\ar[ld]_-{\omega}\ar[rd]^-{\hspace{5pt}\omega_*(\mathbf{1})\otimes -}\\
\QCoh(k)\ar[rr]^-{\omega_*}_-\sim & & \omega_*(\mathbf{1})\text{-}\Mod(T^{\Theta}).
}$$
\end{lem}
\begin{proof}
It suffices to verify that $\omega \colon T^\Theta \to \QCoh^{\LSym}(k)$
satisfies the conditions of Proposition \ref{p1-3}.
As dualizable objects are preserved by symmetric monoidal $\s$-functors, 
$\omega$ preserves compact objects by condition (T1).
Since $T$ is stable by Lemma \ref{l-2-1}, this implies
that its right adjoint $\omega_* \colon \QCoh(k) \to T$ preserves colimits, 
so in particular geometric realizations. Condition $(2)$ of Proposition \ref{p1-3} is
obvious when $x$ is a dualizable object, because we have for any $z \in T$
\begin{align*}
\Map_{T}(z,x \otimes \omega_*(y)) &\simeq \Map_{T}(z\otimes x^\vee,\omega_*(y)) \simeq 
\Map_T(\omega(z) \otimes \omega(x)^\vee,y)\\
&\simeq\Map_T(\omega(z),\omega(x)\otimes y) \simeq \Map_T(z,\omega_*(\omega(x)\otimes y)).\end{align*}
The general case follows by writing $x$ as a colimit of compact
and thus dualizable objects.
\end{proof}
\begin{proof}[Proof of Theorem \ref{t1}]
Let us start by describing the stack 
$\Fib^{\LSym}_*(T^\theta)\simeq B\underline{aut}^\Theta(\omega)$ in more detail.
Let us write
$$
B := \omega_*(\mathbf{1})\in \TCAlg(T^{\Theta})
$$
and let $B^*\colon \Delta_+^{op}\to \TCAlg(T^{\Theta})$ be
the augmented cosimplicial diagram given by the co-nerve of the map
$\mathbf{1}\to B$ (so $B^{-1}=\mathbf{1}$ and $B^0=B$).
We obtain an augmented cosimplicial diagram of $\Theta$-categories 
$$
B^*\text{-}\Mod(T^{\Theta}) \in T^{\Theta}/\tcat
$$
By Proposition \ref{p1-2}, the co-Segal maps exhibit each
$B^n\text{-}\Mod(T^{\Theta})$ as the $(n+1)$-fold coproduct
of $B\text{-}\Mod(T^{\Theta})$ in $T^{\Theta}/\tcat$. 
It follows that the simplicial diagram of stacks of $\Theta$-fiber functors
$$
G_*:=\Fib^{\LSym}(B^*\text{-}\Mod(T^{\Theta}))
$$
is equivalent to the nerve of the map 
$\Fib^{\LSym}(B\text{-}\Mod(T^{\Theta}))\to \Fib^{\LSym}(T^{\Theta})$. 
By Lemma \ref{l-2-2}, this is equivalent to the nerve of the basepoint 
$x\colon \ast\to \Fib^{\LSym}(T^{\Theta})$ corresponding to $\omega$. 
Consequently, we obtain an equivalence of group stacks
$$
\underline{aut}^\Theta(\omega)\simeq \Fib^{\LSym}\big(B^{\otimes 2}\text{-}\Mod(T^{\Theta})\big)=: G
$$
where the group structure on $G$ is encoded by the simplicial diagram $G_*$ 
(in particular, $G_n\simeq G^{\times n}$ via the Segal maps).
The stack $G$ can be identified more explicitly as follows.

\begin{lem}\label{l-2-3}
Let $T^\Theta$ be a neutralized Tannakian $\Theta$-category over $k$. 
Then the stack of $\Theta$-self-automorphisms of its fiber functor $\omega$
is an affine stack, given by
$$
\underline{aut}^\Theta(\omega) \simeq \Spec\, \omega(B).
$$
In addition, the $\LSym$-algebra $\omega(B)$ has positive tor-dimension.
\end{lem}
\begin{proof}[Proof of Lemma \ref{l-2-3}]
Let $\omega_*(\mathbf{1})\otimes B\to\omega_*(\omega(B))$ be the natural map
of $\LSym$-algebras adjoint to the map 
$\omega(\omega_*(\mathbf{1}))\otimes \omega(B)\to \mathbf{1}\otimes \omega(B)$.
This is an equivalence by the projection formula.
Lemma \ref{l-2-2} therefore gives an equivalence
$$
B^{\otimes 2}\text{-}\Mod(T^{\Theta})\simeq \omega(B)\text{-}\Mod(\QCoh(k)).
$$
We now note that the stack of $\Theta$-fiber functors of 
$\omega(B)\text{-}\Mod(\QCoh(k))$ is equivalent to $\Spec\, \omega(B)$
by Proposition \ref{p1-2}.

To see that $\omega(B)$ has positive tor-dimension, let $E$ be a 
coconnective complex of $k$-modules. 
Then the projection formula shows that
$$E \otimes \omega(B)\simeq \omega(\omega_*(E))$$
remains coconnective,
since $\omega$ is t-exact and $\omega_*$ is (hence) left t-exact.
\end{proof}
All in all, we see that $\Fib_*^{\LSym}(T^{\Theta})\simeq BG$ where
$G=\Spec\, \omega(B)$ is the group stack whose group structure
is encoded by the simplicial diagram $G_*$. \\

Let us now prove (1). 
The unit map fits into a natural transformation of 
augmented cosimplicial $\Theta$-categories
$$\xymatrix@C=1.5pc{
T^\Theta\ar[r]\ar[d]_-{\beta_{T^\Theta}} & 
B\text{-}\Mod(T^{\Theta})\ar@<3pt>[r]\ar@<-3pt>[r]\ar[d]^-{\beta_0} &
B^{\otimes 2}\text{-}\Mod(T^{\Theta})\ar[r]\ar@<6pt>[r]\ar@<-6pt>[r]\ar[d]^-{\beta_1}  &
\dots\\
\QCoh^{\LSym}(BG)\ar[r] &
 \QCoh^{\LSym}(k) \ar@<1ex>[r]\ar@<-1ex>[r] &
 \QCoh^{\LSym}(G) \ar[r]\ar@<6pt>[r]\ar@<-6pt>[r] &
  \dots
}$$
The bottom row is a limit diagram by definition. 
For the top row, note that $B$ has positive tor-dimension because
(by Lemma \ref{l-2-3}) $\omega(B)$ has positive tor-dimension
and $\omega$ is $t$-exact and conservative.
In addition, $B\otimes -\colon T\to B\text{-}\Mod(T)$ is conservative,
since it is equivalent to $\omega$ by Lemma \ref{l-2-2}.
An abstract version of flat descent, for which we refer to Proposition \ref{pa},
therefore shows that the top row restricts to an equivalence on
eventually coconnective objects
$$
T^{>-\infty} \simeq \lim_{n\in \Delta}\big(B^*\text{-}\Mod(T^{>-\infty})\big).
$$
 
Since dualizable objects in $T$ are bounded for the $t$-structure,
we therefore obtain a diagram of symmetric monoidal $\s$-categories 
of dualizable objects
$$\xymatrix@C=1.5pc{
T^{rig}\ar[r]\ar[d]_-{\beta_{T^\Theta}} & B\text{-}\Mod(T)^{rig}\ar@<3pt>[r]\ar@<-3pt>[r]\ar[d]^-{\beta_0} & B^{\otimes 2}\text{-}\Mod(T)^{rig}\ar[r]\ar@<6pt>[r]\ar@<-6pt>[r]\ar[d]^-{\beta_1} & \dots\\
\Parf(BG)\ar[r] & \Parf(k) \ar@<1ex>[r]\ar@<-1ex>[r] & \Parf(G) \ar[r]\ar@<6pt>[r]\ar@<-6pt>[r] & \dots
}$$
in which both the top and bottom row are limit diagrams. 

For each $n\geq 0$, the functor 
$\beta_n\colon B^{\otimes 1+n}\text{-}\Mod(T)\longrightarrow \QCoh(G^{\times n})$ 
can be identified as follows. 
By the same projection formula argument as in the proof of Lemma \ref{l-2-3}, there 
is a symmetric monoidal equivalence
$$
B^{\otimes 1+n}\text{-}\Mod(T)\simeq \omega(B)^{\otimes n}\text{-}\Mod.
$$
By construction, this equivalence identifies $\beta_n$ with 
the canonical symmetric monoidal left adjoint $\s$-functor
$$
\OO\otimes_{\omega(B)^{\otimes n}} -\colon \omega(B)^{\otimes n}\text{-}\Mod\longrightarrow \QCoh(G^{\times n}) \simeq \QCoh(\Spec\, \omega(B)^{\otimes n})
$$
coming from the canonical equivalence, as an $E_\s$-algebra, 
between $\omega(B)^{\otimes n}$ and the endomorphisms of the 
unit $\OO\in \QCoh(\Spec\, \omega(B)^{\otimes n})$. Using 
that the right adjoint to the above functor takes global 
sections, one readily sees that the restriction to 
dualizable objects
$\omega(B)^{\otimes n}\text{-}\Mod^{rig} \to \Parf(G^{\times n})$ is fully faithful.

We conclude that each 
$\beta_n\colon B^{\otimes n}\text{-}\Mod(T)^{rig}\to \QCoh(G^{\times n})$ 
is fully faithful. 
In addition, $\beta_0$ is evidently an equivalence. 
It follows that the induced functor on limits
$$
\beta_{T^\Theta} \colon T^{rig}\xrightarrow{\simeq} \Parf(BG)
$$
is an equivalence as well. 
Since $T$ is rigid by (T1), it then follows that $\beta_{T^\Theta}$ induces 
a symmetric monoidal $\s$-categories $T\simeq \IParf(BG)$, proving (1).

We now turn to point (2). 
Condition (G1) for being a Tannakian gerbe is satisfied because 
$\Fib_*^{\LSym}(T^\Theta)$ is pointed and connected by definition. 
Condition (G3) holds by Lemma \ref{l-2-3} and condition (G4) holds by point (1).
It remains to prove that (G2) is also satisfied. 
For any other stack $G$, the inclusion $\Fib_*^{\LSym}(T^\Theta) \subset \Fib^{\LSym}(T^\Theta)$
induces a full embedding on mapping spaces
$$\Map_{\St_k}(G,\Fib_*^{\LSym}(T^\Theta)) \subset \Map_{\St_k}(G,\Fib^{\LSym}(T^\Theta)).$$
If $G$ is furthermore pointed and connected this inclusion fits into a cartesian square
$$\xymatrix{
\Map_{\St_k}(G,\Fib_*^{\LSym}(T^\Theta)) \ar[r] \ar[d] &  \Map_{\St_k}(G,\Fib^{\LSym}(T^\Theta)) \ar[d] \\
B\underline{aut}^{\Theta}(\omega)(k) \ar[r] &
\Fib^{\LSym}(T^\Theta)(k),
}$$
where bottom horizontal morphism is the inclusion of the connected
component containing the global point $\omega \in \Fib^{\LSym}(T^\Theta)(k)$. By the
definition of the stack $\Fib^{\LSym}(T^\Theta)$, the mapping 
space on the top right corner identifies canonically 
$$\Map_{\St_k}(G,\Fib^{\LSym}(T^\Theta)) \simeq \Map_{k/\tcat}(T^\Theta,\QCoh^{\LSym}(G)).$$
As a result, we see that if $f \colon G_1 \to G_2$ is a morphism between pointed and connected stacks,
such that $f^* \colon \QCoh(G_2) \to \QCoh(G_1)$ is an equivalence, then the induced morphism on mapping spaces
$$f^* \colon \Map_{\St_k}(G_2,\Fib_*^{\LSym}(T^\Theta)) \longrightarrow 
\Map_{\St_k}(G_1,\Fib_*^{\LSym}(T^\Theta))$$
is an equivalence as required, 
showing that $\Fib_*^{\LSym}(T^\Theta)$ is always pointed $Q$-local.
\end{proof}

The following corollary gathers the important conclusion of the Theorem \ref{t1}, which states the existence
of Tannakian duals for Tannakian $\Theta$-categories.

\begin{cor}\label{ct1}
Let $T^\Theta$ be a neutralized Tannakian $\Theta$-category over $k$. Then there exists
an affine group stack $G=\Spec\, C$, with $C$ a $k$-linear $LSym$-algebra of positive tor-dimension, 
and an equivalence of symmetric monoidal $\s$-categories
$$\beta : T \simeq \IParf(BG).$$
\end{cor}

\begin{rmk}\label{r1}
The reader will notice that in the proof of Theorem \ref{t1} we have a commutative
diagram of symmetric monoidal $\s$-categories
$$\xymatrix{
T \ar[dr] \ar[r]^-{\simeq}& \IParf(\Fib_*^{\LSym}(T^\Theta)) \ar[d] \\
 & \QCoh(\Fib_*^{\LSym}(T^\Theta))
}$$
where the top horizontal $\s$-functor is an equivalence. 
Moreover, each of the two 
other $\s$-functors possesses a natural lift to a morphism of $\Theta$-categories. However, we have not been 
able to prove that the equivalence $T\simeq \IParf(\Fib_*^{\LSym}(T^\Theta))$
can be enhanced to an equivalence of $\Theta$-categories. This is due to our choice of definition
of $\Theta$-category, which is probably not optimal (see also Remark \ref{r2}).
\end{rmk}

The next corollary investigates the other direction of Tannakian duality, 
and the morphism $\alpha$. We recall that we have introduced the notion of $Q$-local
stacks in the definition of Tannakian gerbes (condition (G2)). It has a refined version 
called $P$-local, which already naturally appears in the setting of schematic 
homotopy theory of \cite{chaff}. A pointed and connected stack $F$ is said to be \emph{pointed $P$-local} 
if for any morphism of
pointed and connected stacks $u \colon G_1 \to G_2$ such that $u^* : \Parf(G_2) \to \Parf(G_1)$ is 
an equivalence, the induced morphism 
$$u^* \colon \Map_{\St_k}(G_2,F) \to  \Map_{\St_k}(G_1,F)$$
is also an equivalence. As perfect complexes are the rigid objects in quasi-coherent complexes, 
being pointed $Q$-local implies being pointed $P$-local. However, the converse does not hold.

\begin{cor}\label{ct2}
Let $F$ be a pointed Tannakian gerbe.
Then, the adjunction morphism
$$\alpha_F \colon F \to \Fib_*^{\LSym}(\IParf{}^{\LSym}(F))$$
induces an equivalence of $\Theta$-categories
$$\alpha_F^* \colon \IParf{}^{\LSym}(\Fib_*^{\LSym}(\IParf{}^{\LSym}(F))) \simeq \IParf{}^{\LSym}(F).$$
In particular, if the stack $F$ is pointed $P$-local, then 
$\alpha_F$ is a retract. If the stacks $F$ and $\Fib_*^{\LSym}(\IParf{}^{\LSym}(F))$ are both 
pointed $P$-local, then it is an equivalence of stacks.
\end{cor}

\begin{proof}
By our Theorem \ref{t1} applied to $\IParf{}^{\LSym}(F)$, the $\s$-functor induces
by pull-back along $\alpha_F$ is an equivalence on Ind-perfect complexes
$$\alpha_F^* \colon \IParf(\Fib_*^{\LSym}(\IParf{}^{\LSym}(F))) \simeq \IParf(F).$$
As the forgetful $\s$-functor $\tcat \to \scat^{\otimes}_{pr}$ is conservative, the corollary follows.
\end{proof}

\begin{rmk}\label{r2}
As a final comment, we see that the results on Tannakian duality are imperfect and do not induce
a nice equivalence between Tannakian $\Theta$-categories and Tannakian gerbes. One reason for this
has been identified in Remark \ref{r1}: the equivalence of symmetric monoidal $\s$-categories
of Corollary \ref{ct2} has not been shown to lift to an equivalence of $\Theta$-categories. We think it is
probably not possible to improve this using the present approach to $\Theta$-categories. The technical
complication comes here from the use of presentable, and thus non-small $\s$-categories everywhere, whereas
the natural approach would be to write everything in terms of $\s$-categories of perfect complexes as opposed
to quasi-coherent sheaves or ind-perfect complexes. The problem with working with small $\s$-categories of
perfect complexes is of course that the monads involved, as the LSym monad does not preserve
perfect complexes globally. 

One possible approach is to consider the LSym monads as graded monads, as
done in $\S 2.1$ for the construction of the $\Theta$-structure in ind-perfect complexes. 
Alternatively, and perhaps more satisfactorily, 
one can develop a theory of $\Theta$-categories in terms of symmetric 
monoidal $\s$-categories equipped with additional symmetric power operations.
Since the derived symmetric powers $\LSym^n$ preserve perfect complexes, 
such a theory would also allow to view $\Parf(X)$ as a $\Theta$-category.
We believe that 
such an approach, allowing for $\Theta$-categories to be small $\s$-categories could reduce the imperfection
of Theorem \ref{t1}. This will be investigated elsewhere.
\end{rmk}

\appendix

\section{Positive tor-dimension descent}

Let $T$ be a stable presentable symmetric monoidal $\s$-category endowed with 
a non-degenerate and multiplicative t-structure. For a morphism of commutative algebras
$u \colon \mathbf{1} \to B$ in $T$, we can consider the co-nerve of $u$, which is a coaugmented 
cosimplicial object
$$CN(u)^* \colon \Delta_+ \longrightarrow \CAlg(T)$$
sending $n\geq -1$ to $B^{\otimes n+1}$. 
Tensoring along $\mathbf{1} \to CN(u)^*$ yields a colimit preserving $\s$-functor
$$\psi \colon T=\mathbf{1}\text{-}\Mod(T) \to \lim_{n\in \Delta}(CN(u)^n\text{-}\Mod(T)).$$
The right adjoint to this $\s$-functor is obtained as the composition
$$\xymatrix{
\lim_{n\in \Delta}\big(CN(u)^n\text{-}\Mod(T)\big) \ar[r] &  T^{\Delta} \ar[r]^-{\lim} & T}$$
where the first $\s$-functor is the levelwise forgetful $\s$-functor to $T$.

We now assume that $B$ has positive tor-dimension. That is, the $\s$-functor $M \mapsto B\otimes M$
preserves the sub-$\s$-category $T^{[0,\s[} \subset T$ of coconnective objects with respect to the
$t$-structure. We will denote by $T^{>-\s}$ the full sub-$\s$-category
of eventually coconnective objects, that is, the union of sub-$\s$-categories
$$\bigcup_{a <0} T^{[a,\s[} \subset T.$$
Similarly, for $B \in \CAlg(T)$ we denote by 
$B\text{-}\Mod(T)^{>-\s} \subset B\text{-}\Mod(T)$ the full sub-$\s$-category
of modules whose underlying objects lie in $T^{>-\s}$.

\begin{prop}\label{pa}
Assume that $B \in \CAlg(T)$ is of positive tor-dimension and that 
$B\otimes - \colon T^{>-\s} \to T^{>-\s}$ is conservative. Then the $\s$-functor
$$\psi \colon T^{>-\s} \to  \lim_{n\in \Delta}\big(CN(u)^n\text{-}\Mod(T)^{>-\s}\big)$$
is an equivalence of $\s$-categories.
\end{prop}

\begin{proof}
Let us first assume that $B$ is augmented via a morphism of
commutative algebras $B \to \mathbf{1}$. In this case, the constant diagram 
$\mathbf{1}$ becomes a homotopy retract of the cosimplicial diagram $CN(u)^*$. The augmentation
induces a retraction $CN(u)^* \to \mathbf{1}$, and the composition
$CN(u)^* \to \mathbf{1} \to CN(u)^*$ is simplicially homotopic to the identity as a morphism 
of cosimplicial objects in $\CAlg(T)$. This implies that 
the cosimplicial diagram of categories $CN(u)^*\text{-}\Mod(T)$ is homotopy equivalent to the constant 
diagram $T$, and thus that we have
$$T \simeq \lim_n CN(u)^n\text{-}\Mod(T).$$
In this special case, the proposition holds even without the restriction to $T^{>-\s}$ and without 
the positive tor-dimension condition, and without any need of a t-structure: it remains true
for any augmented commutative algebra $B$ in a stable presentable symmetric monoidal $\s$-category.

Let us now treat the general case. We first observe that $B \otimes -$ commutes with
limits of diagrams $x_* \colon \Delta \to T^{[a,\s[}$, that is, limits of uniformly bounded below
cosimplicial diagrams. Indeed, in this case, for all $i$ there is a $k$, depending on $a$ and $i$ only, 
such that 
$$\lim_n(x_n) \to \lim_{n\leq k}(x_n)$$
induces an equivalence on $\tau^{\geq i}$. As $B$ is of positive tor-dimension the same
holds for the limit $\lim(B\otimes x_n)$ instead. Because $B\otimes-$ commutes with finite limits, 
we see that the canonical morphism
$$B\otimes \lim_n(x_n) \to \lim_n(B\otimes x_n)$$
is an equivalence on all truncations $\tau^{\geq i}$ for all $i$, and thus must be an equivalence in $T$.

Because $B\otimes -$ is conservative and commutes with the involved limits, 
we can base change to $B$ to prove this last statement. After base change the statement is true
as it is related to the co-nerve of $B \to B\otimes B$ inside $B\text{-}\Mod(T)$, which is augmented via
the multiplication map $B\otimes B \to B$.
\end{proof}

\bibliographystyle{alpha}
\bibliography{Biblio.bib}

\noindent Bertrand Toën, {\sc IMT, CNRS, Université de Toulouse, Toulouse (France)}\\
Bertrand.Toen@math.univ-toulouse.fr \\

\noindent Joost Nuiten, {\sc IMT,  Université de Toulouse, Toulouse (France)}\\
Joost.Nuiten@math.univ-toulouse.fr \\

\end{document}